\documentclass[12pt,a4paper]{amsart}
\usepackage[english]{babel}
\usepackage[a4paper,inner=2.8cm,outer=2.5cm,top=2.8cm,bottom=2.5cm]{geometry}                
\usepackage{amsmath,amssymb,amscd,amsthm,amsfonts,anyfontsize,color,enumerate,fix-cm,layout,lipsum,lpic,mathrsfs,mdwlist,stmaryrd,tensor,tikz,thmtools,xspace}

\usepackage{hyperref}
\usepackage[all]{xy}
\usepackage{mathtools}
\usepackage{xfrac}
\theoremstyle{plain}                 
\newtheorem{theorem}{Theorem}[section]     
\newtheorem{proposition}[theorem]{Proposition}

\theoremstyle{definition}           
\newtheorem{definition}[theorem]{Definition}    
 
\theoremstyle{remark}       
\newtheorem{remark}[theorem]{Remark}    
\hypersetup{colorlinks=true,linkcolor=blue,citecolor=purple,filecolor=magenta,urlcolor=cyan}

\begin{document}

\begin{flushleft}
{\Large
\textbf\newline{The state--sum invariants for virtual knots}
}

\bigskip

 A.A.Kazakov
\\
\bigskip
Moscow State University, Faculty of Mechanics and Mathematics, Centre of Integrable Systems, P. G. Demidov Yaroslavl State University, Center of Fundamental Mathematics, Moscow Institute of Physics and Technology.

\bigskip
 anton.kazakov.4@mail.ru

\section*{Abstract}
  We construct the new non-trivial state--sum invariants for virtual knots and links by a generalization of the powerful Carter--Saito--Jelsovsky--Kamada--Langford  theorem for classical knots. The main result of this work is based on  cohomology quandle theory and   colorings of virtual knot and link diagrams  by quandle elements.
\bigskip

\noindent {Keywords:} Virtual knots, virtual links, quandle, invariant state--sum, cohomology of quandle.

\end{flushleft}

\section{Introducrion}

The theory of virtual knots discovered and described by Kauffman \cite{Kau} arises from the study of  knots in thickened surfaces, which is the natural generalization of classical knot theory. At the present time, this theory is well--developed  and  has many applications in other  fields of mathematics. For instance, virtual knots play an important role in  combinatorics of Gauss diagrams and codes \cite{Kau} \cite{KauM}. Also, virtual knots relate to the theory of virtual manifolds and the generalization of Vitten--Reshetikhin--Turaev invariants \cite{DLK}.

Virtual analogues of knot invariants are in the focus of many research dedicated to virtual knot theory \cite{Kau}. The purpose of this paper is to construct new non--trivial state--sum invariants of virtual knots and links, which are obtained by a generalization of Carter--Saito--Jelsovsky--Kamada--Langford theorem for classical knots \cite{Car}.

\subsection{Kauffman theory of virtual knots and links}
Similarly to the  classical knots there is the diagram knot technique and the complete analogue of the classical  Reidemeister theorem for  virtual knots \cite{Kau}, \cite{KauM}, which play the central role in virtual knot theory.  We define diagrams for virtual knots or links as it is defined in  classical knot or link  theory, however in comparison with the classical case  some intersections of  virtual knot or link diagrams might be  virtual (see Fig.\ref{fig:1}).

\begin{figure}[ht]

\includegraphics[width=0.8\textwidth]{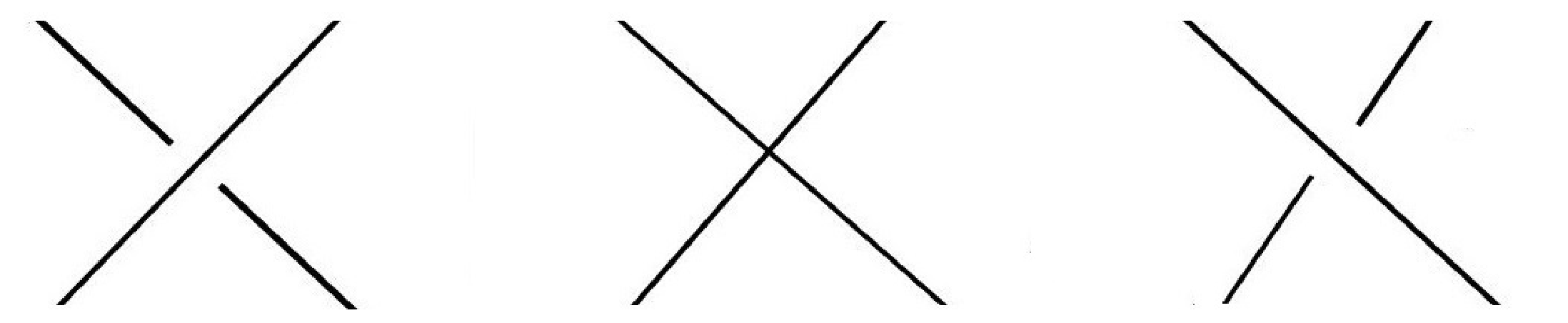}
\caption{The classical  and virtual (in the middle) intersections.}\label{fig:1}
\end{figure}

  According to the virtual Reidemeister theorem \cite{Kau}, \cite{KauM} we consider two virtual knot diagrams equivalent, if one of them can be transformed to another by  isotopy and a finite  sequence of classical Reidemeister moves (see  Fig.\ref{fig:2})  and four additional virtual Reidemeister moves (see  Fig.\ref{fig:3}). 

  \begin{figure}[ht]
\includegraphics[width=0.75\textwidth]{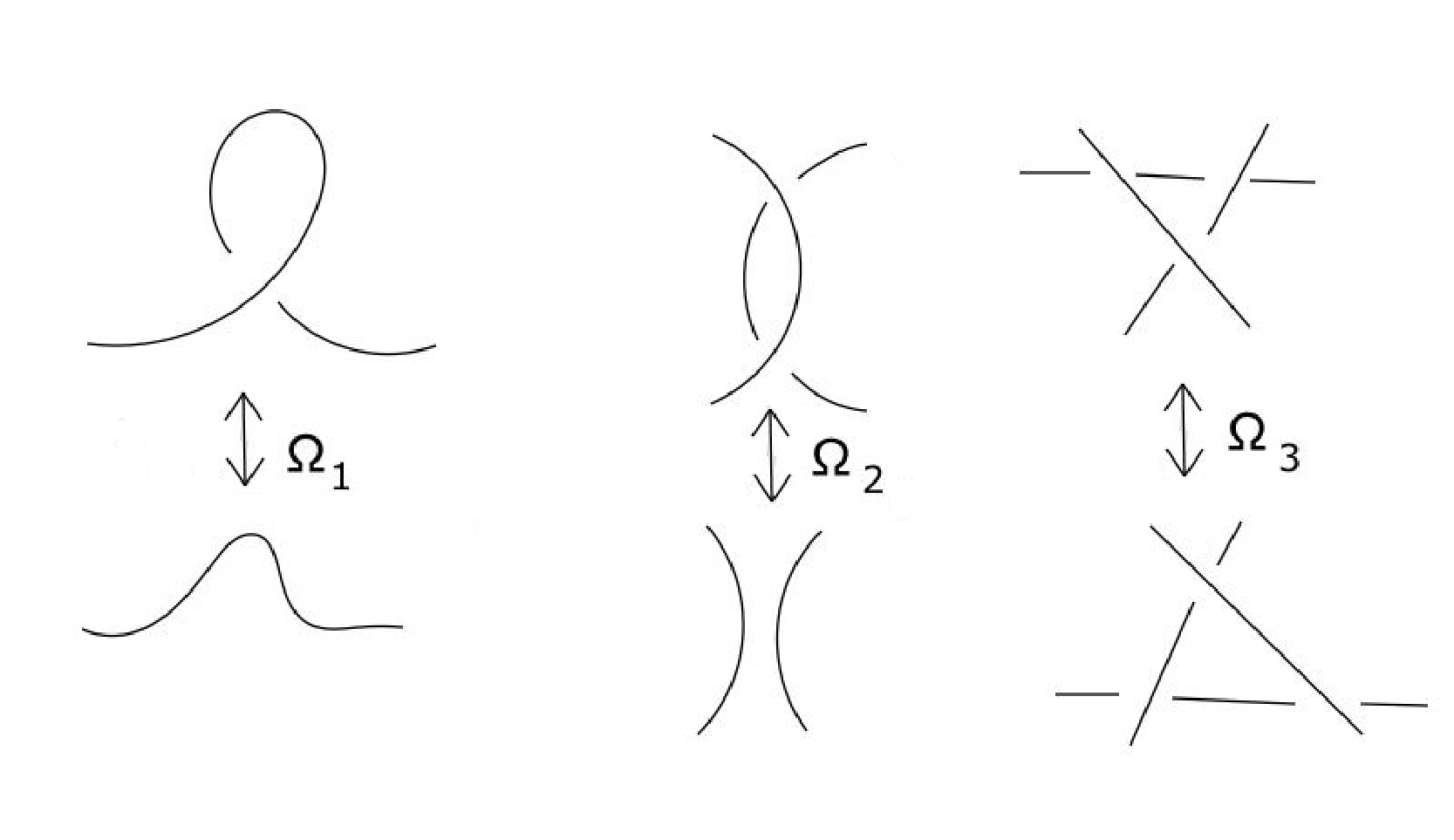}
\caption{The classical Reidemeister  moves.}\label{fig:2}
\end{figure}
 
  \begin{figure}[ht]

\includegraphics[width=0.75\textwidth]{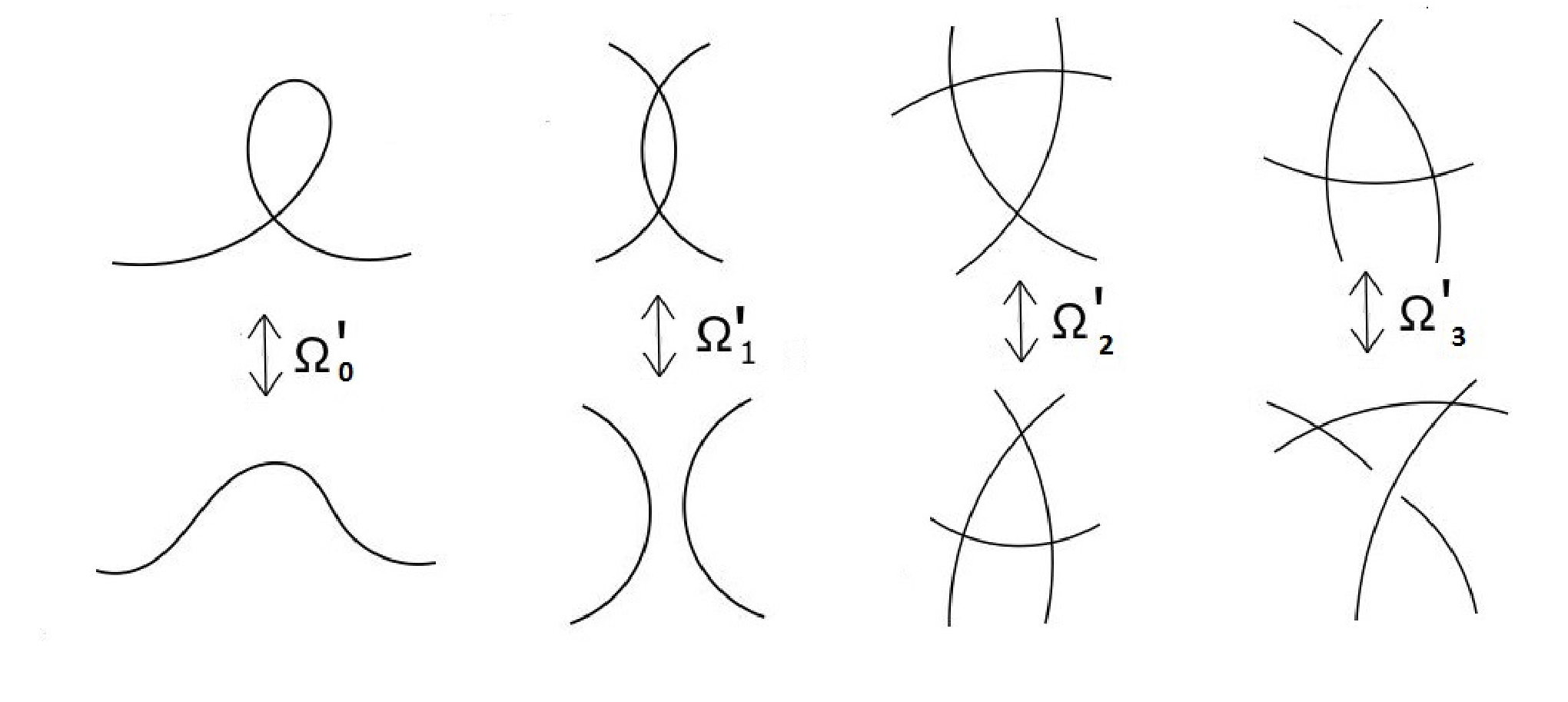}
\caption{The virtual Reidemeister  moves.}\label{fig:3}
\end{figure} 

\subsection{The Carter--Saito--Jelsovsky--Kamada--Langford theorem for classical knots and links}

 In this subsection we briefly examine the construction of a powerful family of knot invariants by Carter, Saito et al. \cite{Car}. Their approach is based on  the following algebraic objects (see \cite{Car}, \cite{Joi}, \cite{Mat} for more details):
 
  \begin{definition} \label{qa}
   A quandle is a set equipped with a binary operation $\ast$ satisfying the following axioms:
 
  \begin{equation} \label{qa_1}
    \forall a \in G : a \ast a = a;  
  \end{equation}
  \begin{equation} \label{qa_2}
    \forall a , b \in G  \: \exists !  x \in G: x \ast a = b;  
  \end{equation}
  \begin{equation} \label{qa_3}
   \forall a , b , c \in G : (a\ast b)\ast c=(a \ast c)\ast (b \ast c).  
  \end{equation}

    \end{definition}

   Also we have to recall the definition  of a quandle 2--cocycle \cite{Car}, \cite{Joi}, \cite{Mat}:
  
  \begin{definition} \label{qcc} A 2--cocycle $\phi$  is a map $\phi : G \times G \rightarrow R$ (here $R$ is an Abelian ring) which satisfies the following identities:

 \begin{equation} \label{qc_1}
    \forall a \in G : \phi (a , a)=1;   
  \end{equation}
  
 \begin{equation} \label{qc_2} 
 \forall a , b , c \in G : \phi(a, b)\phi(a \ast b, c) = \phi(a, c)\phi(a \ast c, b \ast c).
  \end{equation}

  A 2--cocycle $\phi$  is called a coboundary 2--cocycle if the following holds:
  \begin{equation} \label{triv} 
 \phi = \psi (x) \psi^{-1}( x \ast y),
  \end{equation}
  here  $\psi$ is a  map  $\psi: G \rightarrow R.$ 

 2--cocycles $\phi$ and $\phi'$ are called cohomologous if the following holds:
 \begin{equation} \label{coh} 
 \phi = \phi \phi',
  \end{equation}
  here  $\phi$ is a coboundary 2--cocycle.
  
\end{definition}
    
  Now we are ready to describe the Carter--Saito et al. construction for classical knots and links invariants. Let us fix a quandle $G,$ some of its 2--cocycle $\phi$ and consider an oriented diagram of a knot (or a link). Then we color its arcs with elements of $G$. We will say that its arc coloring is possible if the following two rules (see  Fig.\ref{fig:4}) are satisfied for each diagram intersection:
  
  \begin{figure}[ht]
\includegraphics[width=0.55\textwidth]{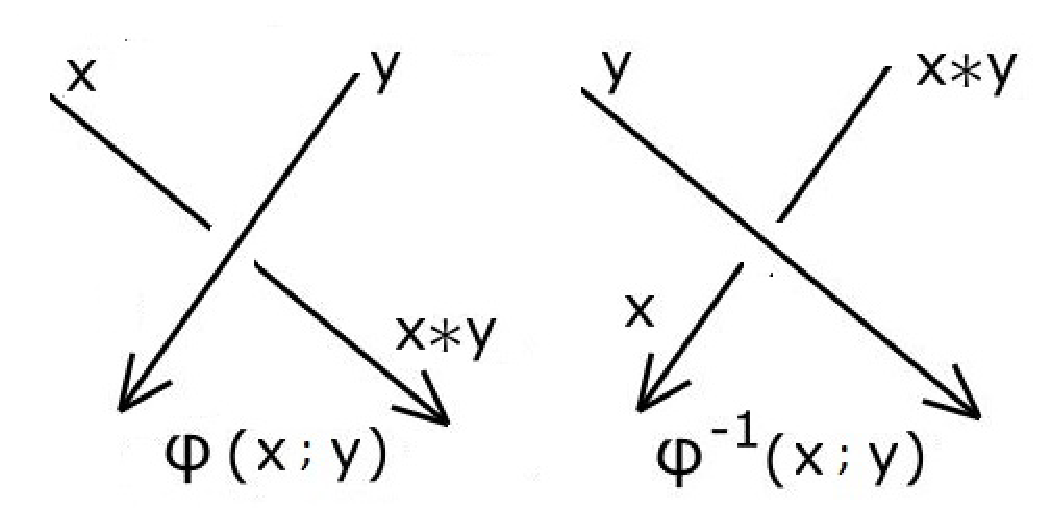}
\caption{The coloring rules.}\label{fig:4}
\end{figure}

\begin{theorem} \cite{Car} \label{cars}
Let us consider a knot (or a link) with a given oriented diagram $D$. Let $C$ be a set of all  possible colorings of the diagram $D.$ For each possible coloring we define the weights of diagram intersections by the 2--cocycle $\phi$ as it is shown in Fig.\ref{fig:4}. Then the following state--sum function is an invariant of the knot (or the link) with the diagram $D$:
$$Z(D , \phi) = \sum\limits_{C} \prod \phi(x , y) ^{\sigma t}, $$ 
 where the product is taken over  all intersections of the diagram $D$ and the sum is taken over the set $C$, signs $\sigma = -1, 1 $ are selected as it is shown in Fig.\ref{fig:4},  $t$ is a formal parameter.
  \end{theorem}
  
  \begin{proof}
 For completeness, we provide the sketch of the proof (for more details see \cite{Car}):

   \begin{figure}[ht]

\includegraphics[width=0.9\textwidth]{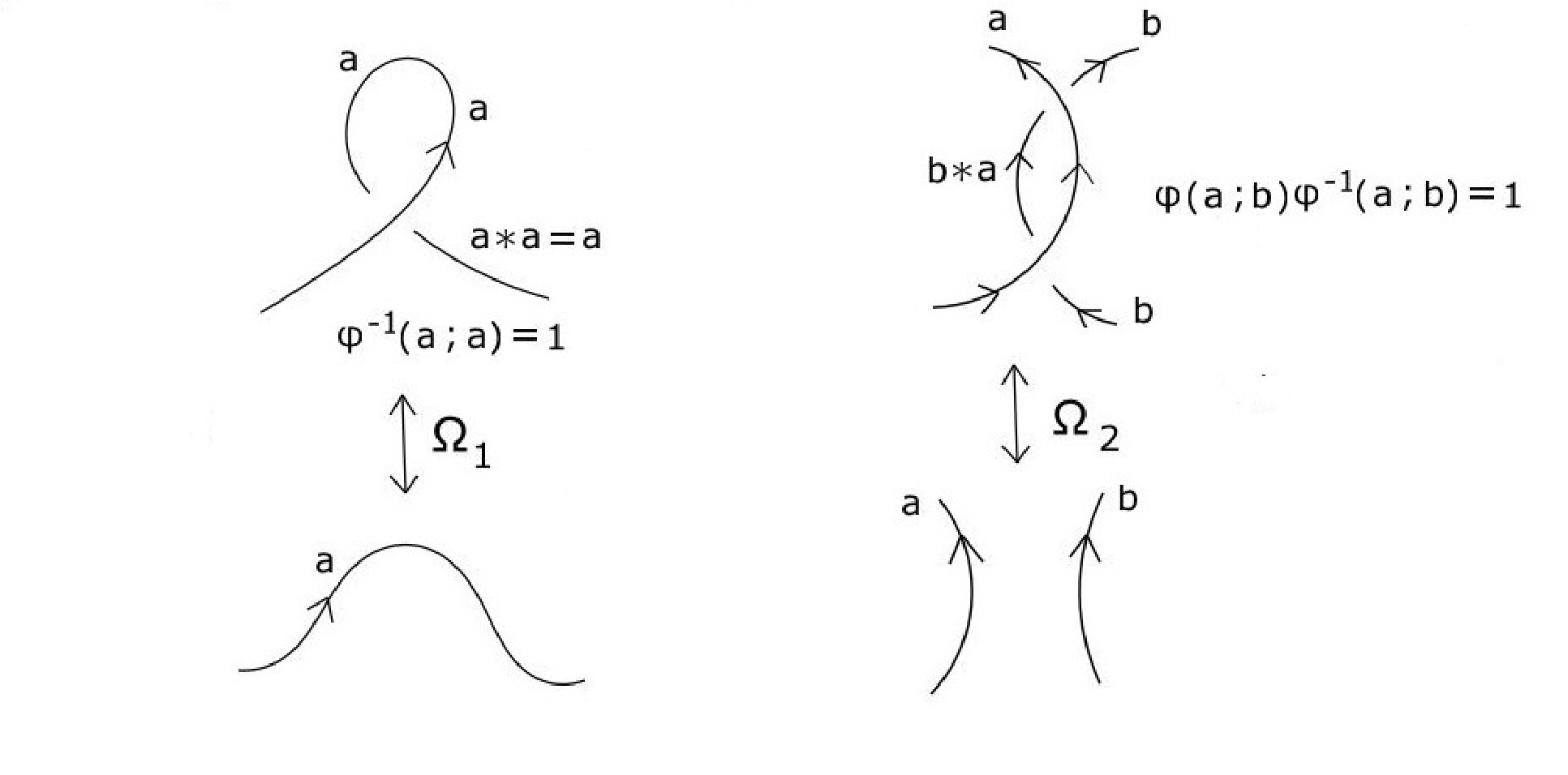}
\caption{The first and second classical moves.}\label{fig:5}
\end{figure}

     1. Consider the first Reidemeister move. Using the first quandle axiom \eqref{qa_1} and the first 2--cocycle identity \eqref{qc_1} we conclude that the state--sum function is preserved under the first Reidemeister move  (see  Fig.\ref{fig:5}).

2. By the same way using the second  quandle axiom  \eqref{qa_2} we conclude that the state--sum function is preserved under the second Reidemeister move (see  Fig.\ref{fig:5}).

  \begin{figure}[ht]
\includegraphics[width=0.9\textwidth]{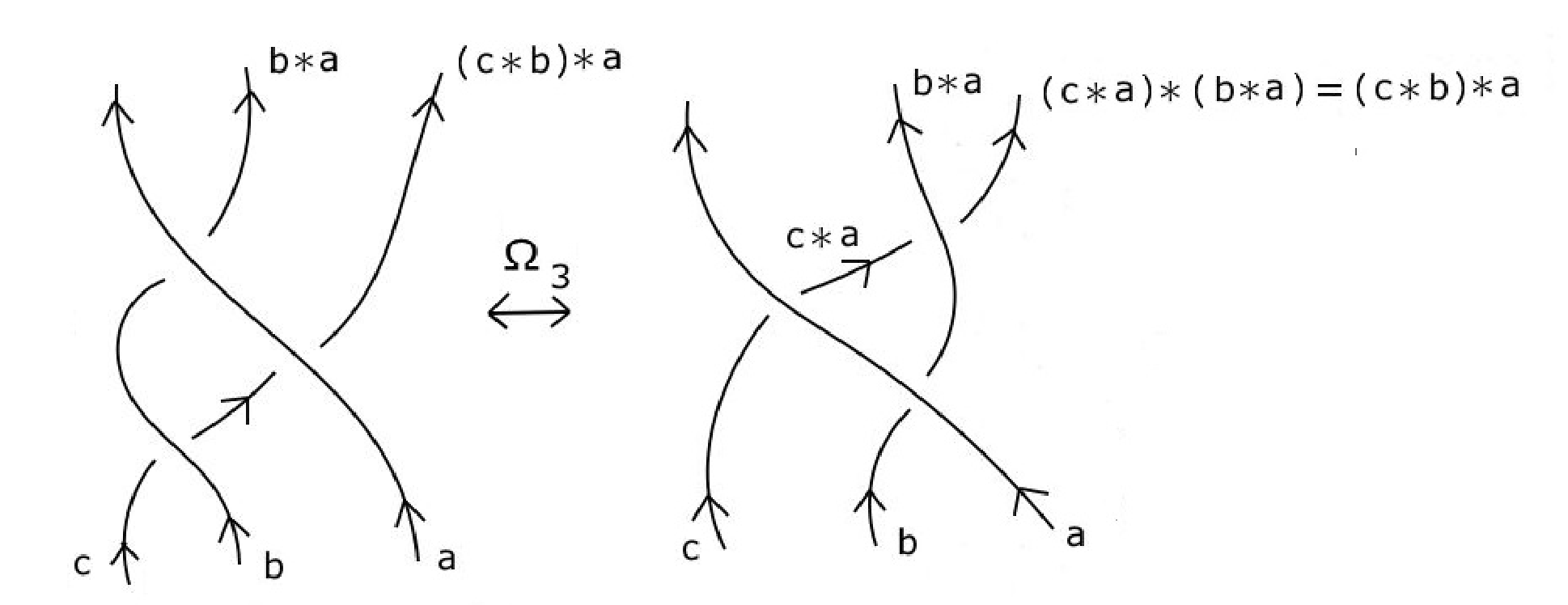}
\caption{The third classical moves.}\label{fig:6}
\end{figure}

 3. The third Reidemeister move remains to be examined. Similarly, using the second and the third quandle axiom \eqref{qa_2}, \eqref{qa_3} and the second 2--cocycle identity \eqref{qc_1}  we conclude that the state--sum function is preserved under the third Reidemeister move (see Fig.\ref{fig:6}).
  \end{proof}

 \begin{remark}
  Formally, other cases of the orientation and the position of the arcs of  diagram $D$ are not analyzed, but they might be checked by the same fashion.
  \end{remark}
  
  \begin{remark}
 An analogue of the Carter--Saito et al. theorem is also holds for knots and links in $\mathbb{RP}^3$  \cite{Gor}. For the definition  of knots and links in $\mathbb{RP}^3$ see  \cite{Man}, \cite{Dro}.  
  \end{remark}
  
\section{Generalization of the Carter--Saito--Jelsovsky--Kamada--Langford  theorem for virtual knots}

In this section we turn to the main result of our article. Consider a virtual knot (or a virtual link) with an oriented diagram $D.$ Then fix  a finite quandle $G,$ its 2--cocycle $\phi$ and an automorphism $f$ of the quandle $G$.
\begin{remark}
Here as usual a quandle automorphism is a bijection $f: G \to G$ which preserves a quandle operation $\forall a, b \in G: f(a \ast b)=f(a)\ast f(b).$   
\end{remark}
Let us color the arcs of a diagram $D$  with the elements of the quandle $G$. We call the coloring possible, if it satisfies the rules demonstrated in  Fig.\ref{fig:4} and Fig.\ref{fig:7}. For each classical intersection we define the weight that is obtained by the 2--cocycle $\phi$ as it is shown in Fig.\ref{fig:4}.  Finally, for each virtual intersection we define the weight is equal to $1$ as it is shown in Fig.\ref{fig:7}.
\begin{remark}
The natural coloring rule for virtual diagrams, which we use in this paper, probably were firstly introduced  in the work \cite{Af}.
\end{remark}

  \begin{figure}[ht]

\includegraphics[width=0.4\textwidth]{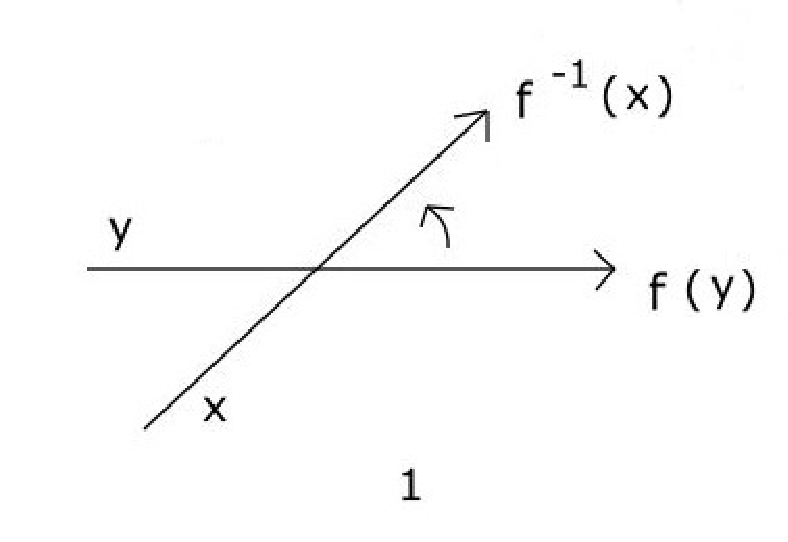}
\caption{The coloring rules for virtual intersection.}\label{fig:7}
\end{figure}

Now we describe how the coloring of a diagram changes under the first three virtual moves (see Fig.\ref{fig:8}). By  simple brute force  it is not difficult to  consider all possible orientations and obtain the following proposition:

\begin{figure}[ht]
\includegraphics[width=0.85\textwidth]{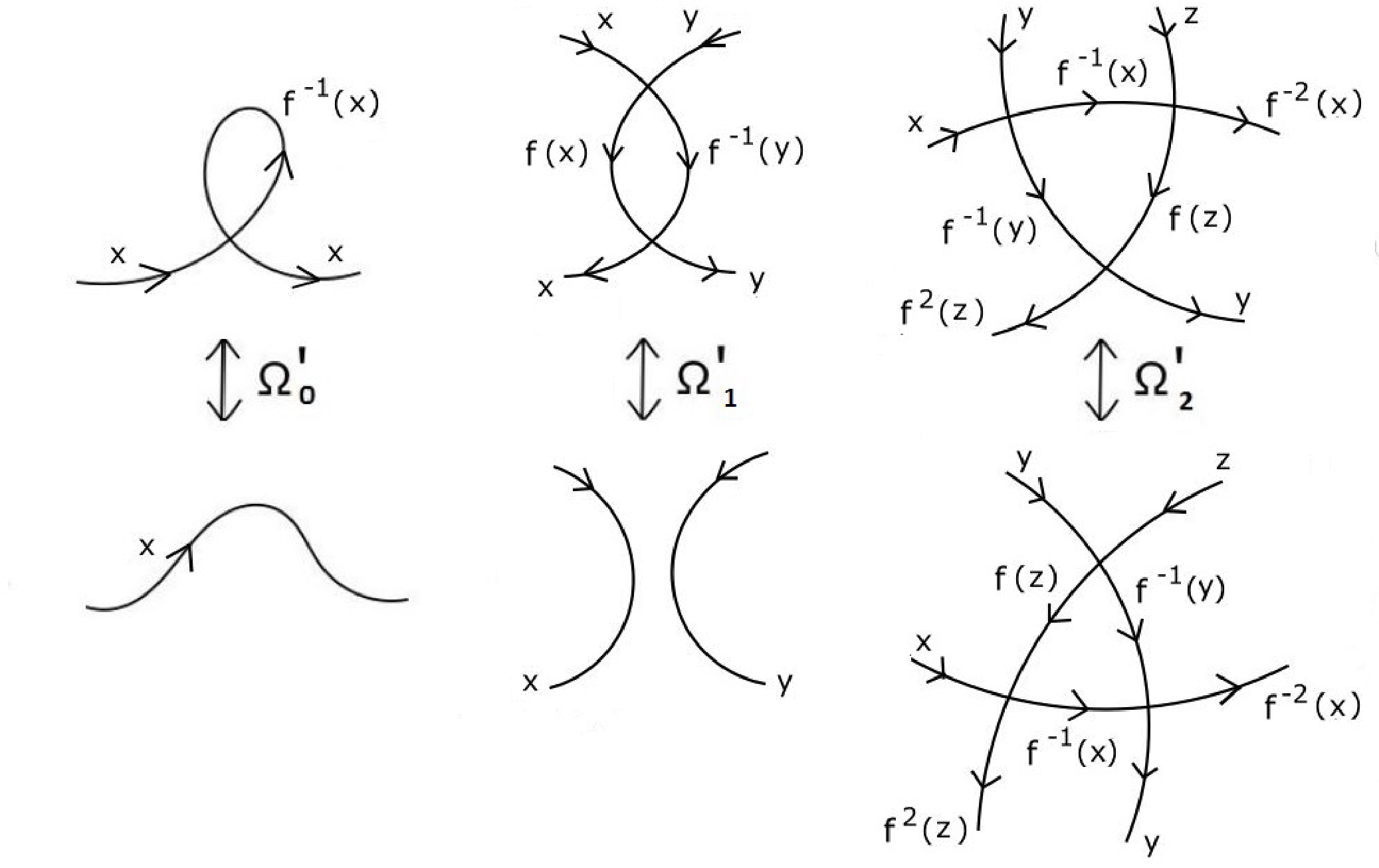}
\caption{The first three virtual moves.}\label{fig:8}
\end{figure}

 \begin{proposition}
 Consider a virtual knot or a link with a diagram $D.$ Then the state--sum function $Z(D, \phi, f) = \sum\limits_{C} \prod \phi(x, y) ^\sigma$ is the invariant of this virtual knot (or link)  under the classical Reidemeister moves and the first three virtual Reidemeister  moves. Where the product is taken over all classical intersections of the given diagram $D$, and the sum is taken over the set of all possible colorings $C$, $\sigma = -1, 1$ is selected as shown in Fig.\ref{fig:4}.
   \end{proposition}

   Now we present how the coloring of a virtual knot (or a link) diagram changes under the move $\Omega'_3 $  (see Fig.\ref{fig:9}) to obtain the following theorem:
  
   \begin{figure}[ht]
\includegraphics[width=0.4\textwidth]{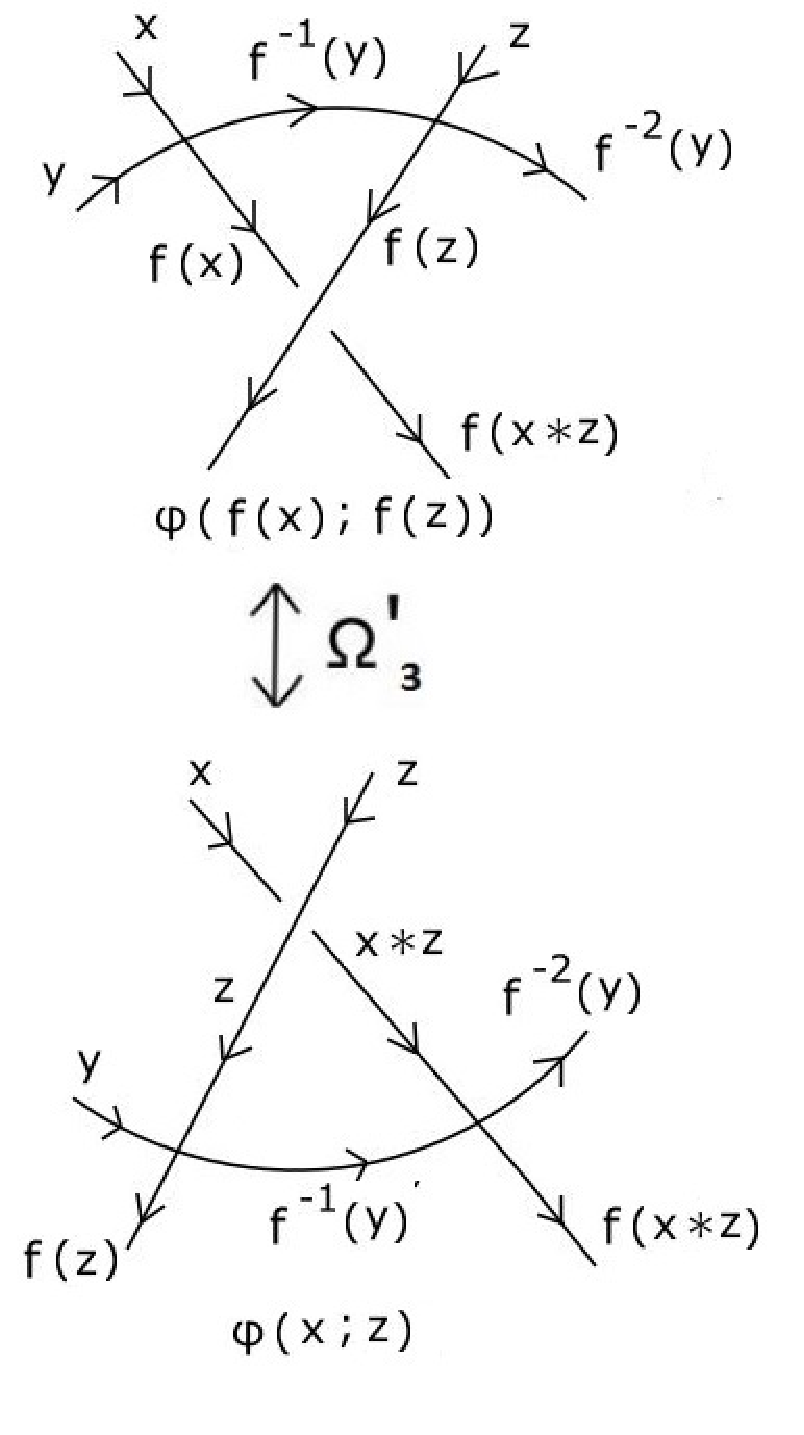}
\caption{ The $\Omega'_3 $ move.}\label{fig:9}
\end{figure}

  \begin{theorem}
 
  1. The function $$Z_1(D , \phi  , f) = \prod\limits_{C} \prod \phi(x , y) ^{\sigma t}$$ is the invariant of a virtual knot (or a virtual link) with a given diagram $D$ under all classical and virtual  moves, where the second product is taken over all classical intersections of the given diagram, and the first product is taken over the set of all possible colorings $C$. We call the function $Z_1$  a state--weight.
  
 2. Consider a fixed automorphism $f$ of a quandle $G$ and  its fixed 2--cocycle $\phi$. Then if $f$  and $\phi$ are aligned with each other -- it means that the following  holds:  $\forall a , b \in G : \phi(a, b)=\phi (f(a), f(b))$, the function $$Z_2(D , \phi , f) =\sum\limits_{C} \prod \phi(x , y) ^{\sigma t}$$ is the invariant of a virtual knot (or link) with a diagram $D$ under all classical and virtual  moves, where the product is taken over all classical intersections of the diagram $D$, and the sum is taken over the set of all possible collorings $C$. We  call the function $Z_2$ a state--sum. 
  
  3. If 2--cocycles $\phi$ and $\phi^{'}$ are cohomologous (see \ref{coh}), then $Z_1(D , \phi  , f)=Z_1(D , \phi^{'} , f)$.
  \end{theorem}
  \begin{proof}

  1. It suffices to prove that $Z_1$ is invariant under $\Omega'_3 $. Values of $Z_1$ before and after  $\Omega'_3 $ move are different on factors which are equal to $P_1:=\prod \phi(x_i , y_i)$ and  $P_2:=\prod \phi(f(x_i) , f(y_i))$ correspondingly. 

  Let us consider $A$ -- an arbitrary possible coloring of $D$, define  the colloring $f(A)$ as follows: if in the colloring  $A$  an arc is colored by an element $a$, then in the coloring  $f(A)$ the same arc is colored by the element $f(a).$ As soon as  $f$ is an  automorphism we conclude that if $A$ is a possible coloring of $D$, then $f(A)$  is a possible coloring  of $D$ too. Moreover, it is easy to see that $f$ acts on the set of all possible collorings as a permutation. So, we obtain that $P_1$ and $P_2$ are different only on an order of factors,  that completes the proof.

  2. The proof is similar to the proof of the Carter--Saito et al. theorem for classical knots  (see Fig.\ref{fig:5}, \ref{fig:6}, \ref{fig:8} and \ref{fig:9}).

  3. Let us introduce an equivalence relation on the set $C$ of all possible collorings: two possible collorings $A_1$ and $A_2$ are equivalent if $A_2=f^k(A_1)$ for some $k \in \mathbf{N}.$  Since the quandle $G$ is finite, then there exists a minimal $n$ such that $id=f^{n}$, and therefore the introduced  relation is  the true equivalence relation. Denote by $C_A$ an equivalence class of a colloring $A,$ using this notation we can rewrite $Z_1$ as follows:
   $$Z_1(D , \phi  , f) = \prod\limits_{C_A} \Bigr(\prod \limits_{B \in C_A} \prod \phi(x , y) ^{\sigma t}\Bigl),$$
    where the third product is taken over all classical intersections of $D$, the second product is taken over all possible collorings $B$ belonged to a selected  equivalence class $C_A$ and the first --  over all  equivalence classes of possible collorings. 
  
 Now, we are ready to prove the third statement. Easy to see, that for this goal it suffices to show that $Z_1(D , \phi  , f) =1$  for any coboundary 2--cocycle  $\phi (x , y)= \psi (x) \psi^{-1}( x \ast y)$ (see \ref{triv}). 
 
 So, let us fix a   coboundary 2--cocycle and consider a long virtual arc of $D$ (i.e. a part of the $D$ between two upper intersections which contains only virtual and bottom intersections, see Fig.\ref{fig:10}).  By direct computation a contribution of a possible colloring $A$  to the $Z_1$ on the long virtual arc  is equal to $\psi^{-1} (\omega) \psi( f^{m}(\omega))$ (here $m$ is some integer which depends only on number of virtual intersections crossed the long virtual arc).  A contribution of all possible collorings $B \in C_A$ to the $Z_1$ on the long virtual arc  is equal to  
 $$\psi^{-1} (\omega) \psi\Bigr( f^{m}(\omega)\Bigl)\psi^{-1} \Bigr(f(\omega)\Bigl) \psi\Bigr( f^{m+1}(\omega)\Bigl)\dots \psi^{-1} \Bigr(f^{|C_A|-1}(\omega)\Bigl) \psi\Bigr( f^{|C_A|-1+m}(\omega)\Bigl). $$
 
 Rewriting the last expression as 
 $$\psi^{-1} (\omega) \psi^{-1} \Bigr(f(\omega)\Bigl) \dots \psi^{-1} \Bigr(f^{|C_A|-1}(\omega)\Bigl) \psi\Bigr( f^{m}(\omega)\Bigl)\psi\Bigr( f^{m+1}(\omega)\Bigl) \dots\psi\Bigr( f^{|C_A|-1+m}(\omega)\Bigl) $$ and taking into account that $f$ acts on $C_A$ as a permutation, we obtain that this product is equal to $1$, that completes the proof.
  \end{proof}

\begin{figure}[ht]
\includegraphics[width=0.8\textwidth]{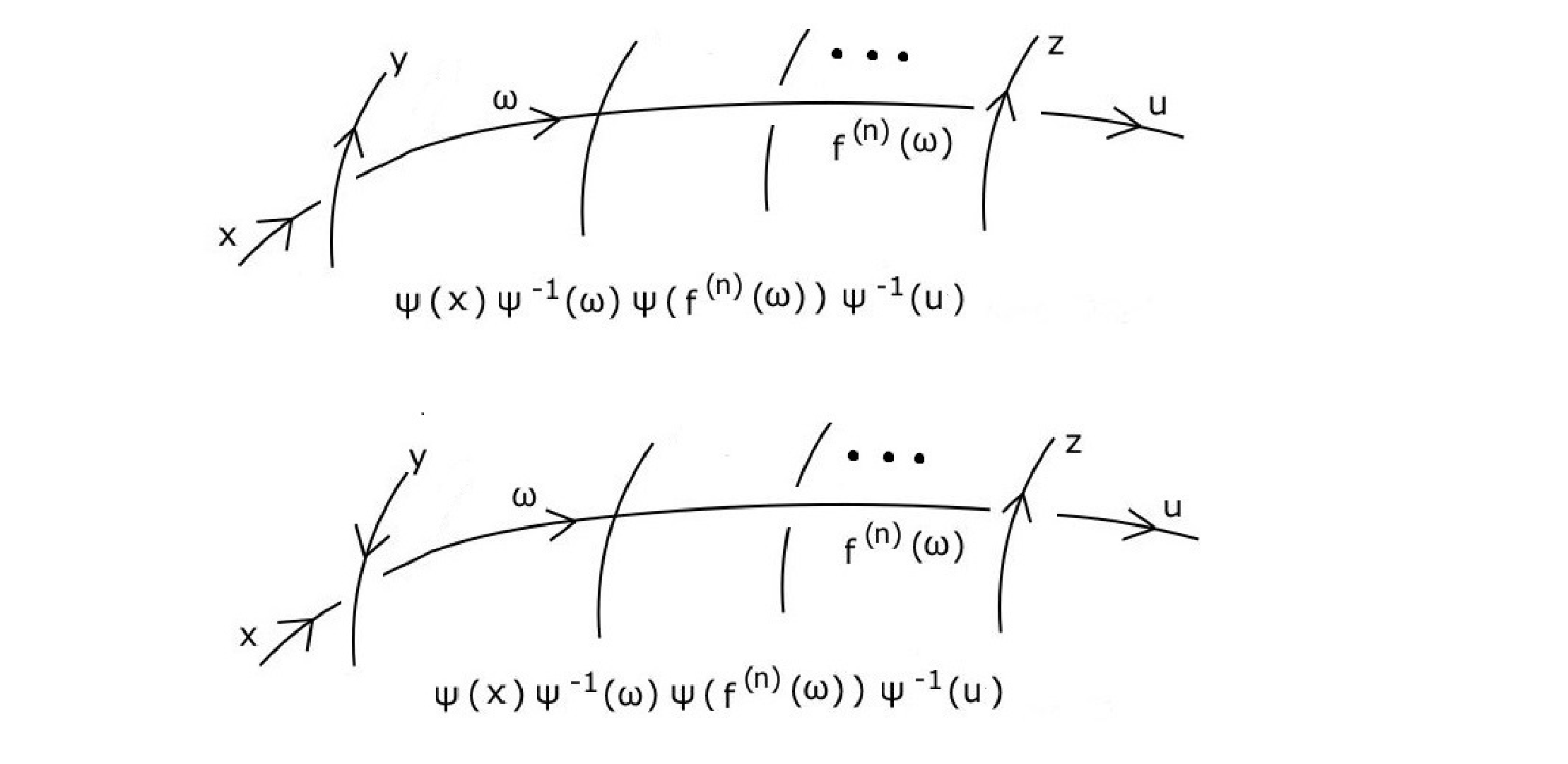}
\caption{A long virtual arc.}\label{fig:10}
\end{figure}

In conclusion of this section, we formulate the following theorem which are obtained as the simple corollary of the speculations above:

  \begin{theorem}
  1. The function $Z_3(D , \phi) =\sum \prod\limits_{C} \prod \phi(x , y) ^{\sigma t}$ is the invariant of a virtual knot (or  virtual link) under all classical and virtual  moves, where the sum  is taken over all automorphism of $G$.

2. The function $$Z_4(D , \phi  , f) = \sum\limits_{C_A} \Bigr(\prod \limits_{B \in C_A} \prod \phi(x , y) ^{\sigma t}\Bigl),$$
    where the second product is taken over all classical intersections of $D$, the first product is taken over all possible collorings $B$ belonged to a selected  equivalence class $C_A$ and the sum --  over all  equivalence classes of possible collorings.

3. If 2--cocycles $\phi$ and $\phi^{'}$ are cohomologous (see \ref{coh}), then $Z_4(D , \phi  , f)=Z_4(D , \phi^{'} , f)$.
  
  \end{theorem}

 \section{Examples} Consider the quandle $Q=\mathbb{Z}_4$ with the operation $\ast: i\ast j=2j-i$ $(mod$ $ 4)$. We have the non-trivial 2--cocycle $\phi: Q \rightarrow \mathbb {Z}$, which is defined as follows: $\phi(a_1, b_1)=\phi(a_1, b_2)=t$ otherwise $\phi =1$, where $t$ is a generator under multiplication of $\mathbb {Z}$, and $a_1=0$, $b_1=1$, $a_2=2$, $b_2=3$. It is easy to verify that the automorphism $f(b)=b \ast a_1$ is aligned with the given  2--cocycle $\phi$.

Now we calculate the state--sum related with the quandle $Q$ and the automorphism $f$ for the links as shown in Fig.\ref{fig:11}.

 \begin{itemize}
     \item[1.] For the first example the state--sum is equal to $8(1+t)$.
     \item[2.] For the second example the state--sum is equal to $8$.
     \item[3.] For the third example the state--sum is equal to $4(1+t)$.
     \item[4.] For the fourth example the state--sum is equal to $4(1+t)$.
 \end{itemize}

  \begin{figure}[ht]
\includegraphics[width=0.75\textwidth]{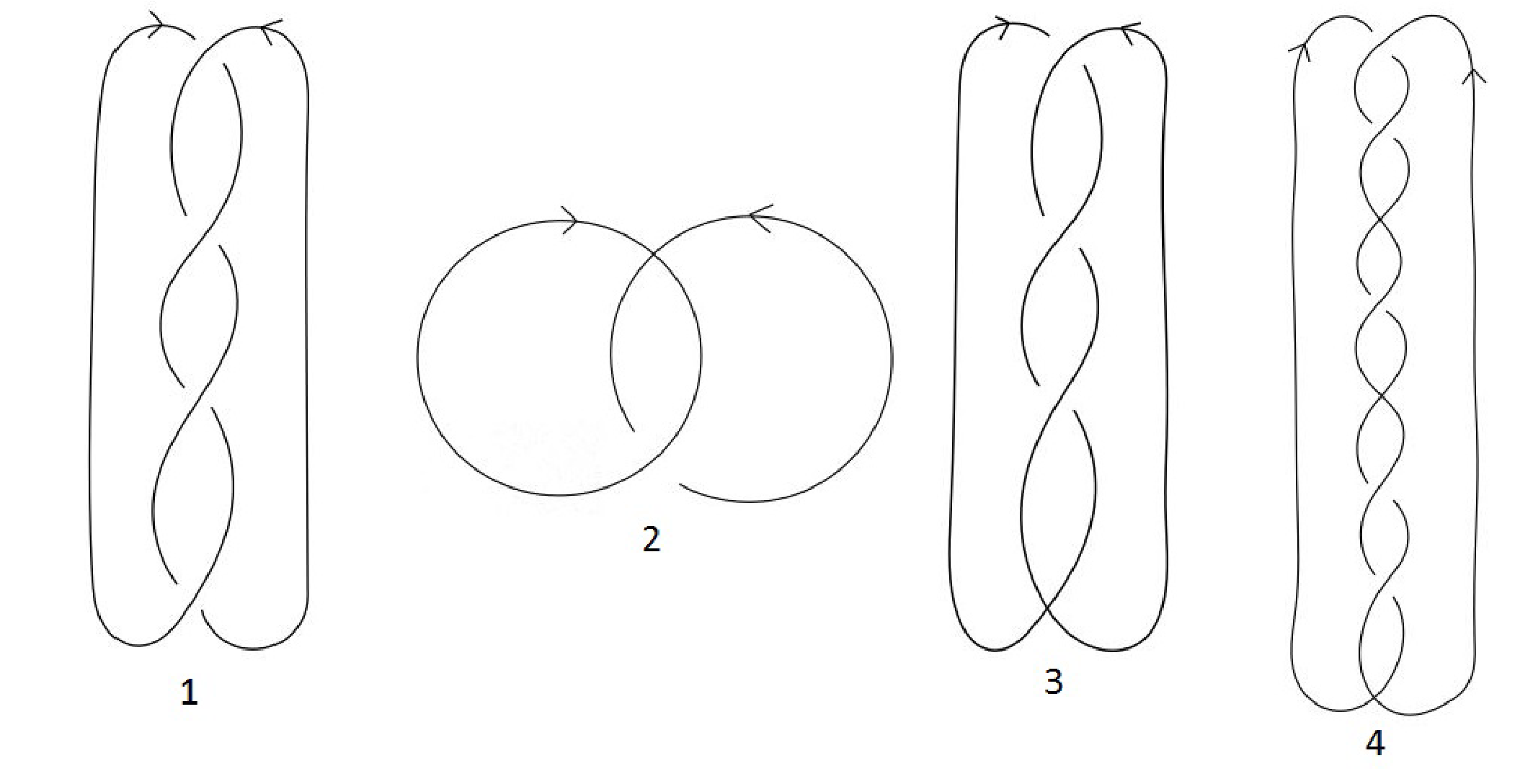}
\caption{Examples.}\label{fig:11}
\end{figure}

\subsection*{Acknowledgments}
The author is  grateful to  Dmitry V. Talalaev (Moscow State University, Faculty of Mechanics and Mathematics) for the support and fruitful discussions,  Evgenii Pavlov (National Research University ``Higher School of Economics", Faculty of Mathematics) and Ammar Basheer (Ural Federal University, Ekaterinburg Department of Humanities)  for careful reading of the article and helpful remarks.

The work was partially supported by the Basis foundation, the grant Leader (Math)
20-7-1-21-1.

\end{document}